\documentclass{amsart}
\usepackage{amssymb,amscd}

 \newtheorem{thm}{Theorem}[section]
 
 \newtheorem{lemma}[thm]{Lemma}
 \newtheorem{prop}[thm]{Proposition}
 \theoremstyle{definition}
 \newtheorem{defn}[thm]{Definition}
 \newtheorem{exmp}[thm]{Example}
 \theoremstyle{remark}
 \newtheorem{rem}[thm]{Remark}

 \newtheorem{question}[thm]{Question}
 \numberwithin{equation}{subsection}
 \newtheorem{ack}{Acknowledgment}
 
\newcommand{\UU}{\text{$\mathcal{U}$}}
\newcommand{\FF}{\text{$\mathcal{F}$}}
\newcommand{\GG}{\text{$\mathcal{G}$}}

\newcommand{\sg}{\text{$\sigma$}}

\newcommand{\Cat}{\operatorname{Cat}}

\newcommand{\Crit}{\operatorname{Crit}}
\newcommand{\id}{\operatorname{id}}

        \newcommand{\field}[1]{\text{$\mathbb{#1}$}}
        \newcommand{\N}{\field{N}}
        \newcommand{\Z}{\field{Z}}
        
        \newcommand{\R}{\field{R}}

\newdimen\theight
\def\TeXref#1{%
             \leavevmode\vadjust{\setbox0=\hbox{{\tt
                     \quad\quad  {\small \textrm #1}}}%
             \theight=\ht0
             \advance\theight by \lineskip
             \kern -\theight \vbox to
             \theight{\rightline{\rlap{\box0}}%
             \vss}%
             }}%

\begin{document}

\title{Tangential category and critical point theory}

\author{Carlos Meni\~no Cot\'on}

\address{Departamento de Xeometr\'{\i}a e Topolox\'{\i}a\\
         Facultade de Matem\'aticas\\
         Universidade de Santiago de Compostela\\
         15782 Santiago de Compostela}

\email{carlos.meninho@gmail.com}

\thanks{Supported by MICINN (Spain): FPU program and Grant MTM2008-02640}


\begin{abstract}
Classical Ljusternik-Schnirelmann category is upper bounded by the number of critical points of any bounded from below differentiable functions of Palais-Smale type. Here we achieve an adaptation of this result for the tangential category of foliations. We introduce a weaker type of Palais-Smale function, obtaining a slight improvement in the classical theorem of critical points.
\end{abstract}

\maketitle

\section{Introduction}

The LS category is a homotopy invariant given by the minimum number of
open subsets contractible within a topological space needed to
cover it. It was introduced by L.~Lusternik and L.~Schnirelmann in 1930's in the setting of variational problems. Many variations of this invariant has been given. In
particular, H.~Colman and E.~Mac\'ias introduced a tangential version for
foliations, where they use leafwise contractions to transversals
\cite{HellenColman, Macias}. 

Many of the properties of the usual category (see \cite{James}) were adapted to tangential category like lower bounds given by using cuplengths of rings in algebraic topology, a global upper bound given by the dimension of the foliation or a upper semicontinuity property when it is evaluated in a space of foliations on a fixed compact manifold \cite{Macias,Hurder,Vogt-Singhof}. However the classical theorem relating LS category with the number of critical points of a differentiable function is still missing \cite{Schwartz}.

In this paper we extend the definition of tangential category to Hilbert laminations. We observe that under general conditions we can embed this kind of laminations into a global Hilbert space. Therefore, we get some control in the transverse topology.

Considering a leafwise differentiable function, we restrict to the case of functions with leafwise isolated critical points. So the set of critical points is a transverse object. In this moment we define the critical sets associated to the function. They are related by a order relation given by the leafwise gradient flow induced by the function.

Remark that the classical theorem of J. Schwartz \cite{Schwartz} holds for bounded from below Palais-Smale functions. In this paper we introduce a new weaker notion of Palais-Smale map introducing topological restrictions on the critical sets.

In addition to its own interest, it si expected applications to variational problems \cite{Ballman,Lusternik-Schnirelmann}.

\section{Definition of the tangential category of laminations}

We refer to \cite{Candel-Conlon} for the basic notion and definitions about laminations, foliated chart, foliated atlas, holonomy pseudogroup and transverse invariant measure. They are recalled here in order to fix notations.

A {\em Polish space\/} is a second countable and complete metrizable space. Let $X$ be a Polish space. A {\em foliated chart\/} in $X$ is a pair $(U,\varphi)$ such that $U$ is an open subset of $X$ and $\varphi:U\to B^n\times S$ is an homeomorphism, where $B^n$ is an open ball of $\R^n$ and $S$ is a Polish space. The sets $B^n\times\{\ast\}$ are called the {\em plaques\/} of the chart, and the sets of the form $\varphi^{-1}(\{\ast\}\times S)$ are called the {\em associated transversals\/}. The map $U\to S$ is called the projection associated to $(U,\varphi)$. A {\em foliated atlas\/} is a family of foliated charts, $\{(U_i,\varphi_i)\}_{i\in I}$, that covers $X$ and the change of coordinates between the charts preserves the plaques; i.e., they are locally of the form $\varphi_i\circ\varphi_j^{-1}(x,s)=(f_{ij}(x,s),g_{ij}(s))$; these maps $g_{ij}$ form the holonomy cocycle associated to the foliated atlas. A {\em lamination\/} $\FF$ on $X$ is a maximal foliated atlas satisfying the above hypothesis. The plaques of the foliated charts of a maximal foliated atlas form a base of a finer topology of $X$, called the {\em leaf topology\/}. The connected components of the leaf topology are called the {\em leaves\/} of the foliation. The {\em dimension\/} of the lamination is the dimension of the plaques when all of them are open sets of the same Euclidean space.

A foliated atlas, $\UU=\{(U_i,\varphi_i)\}_{i\in \N}$, is called {\em regular\/} if it satisfies the following properties:
\begin{enumerate}
  \item [(a)] It is locally finite.
  \item [(b)] If a plaque $P$ of any foliated chart $(U_i,\varphi_i)$ meets another foliated chart $(U_j,\varphi_j)$, then $P\cap U_j$ is contained in only one plaque of $(U_j,\varphi_j)$.
  \item [(c)] If $U_i\cap U_j\neq \emptyset$, then there exists a foliated chart $(V,\psi)$ such that $U_i\cup U_j\subset V$, $\varphi_i=\psi|_{U_i}$ and $\varphi_j=\psi|_{U_j}$.
\end{enumerate}
Any topological lamination admits a regular foliated atlas, also we can assume that all the charts are locally compact. For a regular foliated atlas $\UU=\{(U_i,\varphi_i)\}_{i\in \N}$ with $\varphi_i:U_i\to B_{i,n}\times S_i$, the maps $g_{ij}$ generate a pseudogroup on $\bigsqcup_iS_i$. Holonomy pseudogroups defined by different foliated atlases are equivalent in the sense of \cite{Haefliger}, and the corresponding equivalence class is called the {\em holonomy pseudogroup\/} of the lamination; it contains all the information about its transverse dynamics. A {\em transversal\/} $T$ is a topological set of $X$ so that, for every foliated chart, $(U,\varphi)$, the corresponding projection restricts to a local homeomorphism $U\cap T\to S$. A transversal is said to be {\em complete\/} if it meets every leaf. On any complete transversal, there is a representative of the holonomy psudogroup which is equivalent to the representative defined by any foliated atlas via the projection maps defined by its charts.

\begin{rem}
Observe that the tangential model of the charts could be changed in order to give a more general notion of lamination: instead of taking open balls of $\R^n$ as $B^n$, we could take connected and locally contractible Polish spaces or separable Hilbert spaces. Also, it is possible to define the notion of $C^r$ foliated structure by assuming that the tangential part of changes of coordinates are $C^r$, with the leafwise derivatives of order $\le r$ depending continuously on the transverse coordinates. We can speak about regular atlases in Hilbert laminations but we cannot assume that its foliated charts are locally compact.
\end{rem}

Let us recall the definition of tangential category \cite{HellenColman,Macias}.
A lamination $(X,\FF)$ induces a foliated
measurable structure $\FF_U$ in each open set $U$. The space $U\times\R$ admits an
obvious foliated structure $\FF_{U\times\R}$ whose leaves are
products of leaves of $\FF_U$ and $\R$. Let $(Y,\GG)$ be another measurable lamination. A foliated map $H:\FF_{U\times\R}\to \GG$ is called a {\em tangential
homotopy\/}, and it is said that the maps $H(\cdot,0)$ and $H(\cdot,1)$ are {\em tangentially homotopic\/}. We use the term {\em
tangential deformation\/} when $\GG=\FF$ and $H( -
,0)$ is the inclusion map of $U$. A deformation such that $H( - ,1)$ is constant on the leaves of
$\FF_U$ is called a {\em tangential contraction\/} or an \FF-{\em
contraction\/}; in this case, $U$ is called a {\em tangentially categorical\/} or \FF-{\em categorical\/} open set. The {\em tangential category\/} is the lowest number of
categorical open sets that cover the measurable lamination. On one leaf foliations, this definition agrees with the classical category. The category of \FF\ is
denoted by $\Cat(\FF)$. It is clear that it is a tangential homotopy invariant.

Now, we introduce the relative category that will be useful for further applications.

\begin{defn}
Let $U\subset X$ be an open subset. The {\em relative category\/} of $U$, $\Cat(U,\FF)$, is the minimum number of \FF-categorical open sets in $X$ that cover $U$.
\end{defn}

\begin{rem}
Clearly, $\Cat(U,\FF)\leq\Cat(\FF_{U})$.
\end{rem}

\begin{prop}[Subadditivity of the relative category]\label{p:subaddtivity of the relative category}
Let $\{U_i\}_{i\in\N}$ be a countable family of open subsets of $X$. Then
\[
  \Cat\left(\bigcup_iU_i,\FF\right)\leq\sum_i\Cat(U_i,\FF)\;.
\]
\end{prop}

The following result is about the structure of a tangentially categorical open set.

\begin{lemma}[Singhof-Vogt~\cite{Vogt-Singhof}]\label{l:vogt0}
Let $\FF$ be a foliation of dimension $m$ and codimension $n$ on a
manifold $M$, let $U$ be an \FF-categorical open set, let $x\in
U$, let $D\subset U$ be a transverse manifold of dimension $n$, and
suppose that $x$ belongs to the interior of $D$. Then there exists
a neighborhood $E$ of $x$ in $D$ such that any leaf of $\FF_{U}$
meets $E$ in at most one point.
\end{lemma}

Observe that this lemma extends to the case of Hilbert laminations with the same proof. Hence, if the ambient space is separable, the final step of a tangential contraction is a countable union of local transversals of the foliation.

For the known properties, bounds and computations (in finite dimension) for the tangential category we refer to \cite{Macias,Hurder,Vogt-Singhof}.

\section{Previous facts to the main theorem}

We work in the following with Hilbert  laminations. Thus, the foliated charts are homeomorphisms to $H\times P$, where $H$ is a separable Hilbert space and $P$ is a Polish space. The ambient space of the foliation is a Polish space, and therefore we can work with countable foliated atlases. We hope that the work of this section will be useful to study laminated versions of variational problems where the classical Lusternik-Schnirelmann category was applied .

Moreover we consider $C^2$ Hilbert laminations; i.e., the change of foliated coordinates are leafwise $C^2$ whose tangential derivatives of order $\le2$ are continuous on the ambient space. Also, any lamination is assumed to have a locally finite atlas such that each plaque of every chart meets at most one plaque of any other chart. We consider a leafwise Riemannian metric  so that its tangential derivatives of order $\le2$ are continuous on the ambient space; thus each leaf becomes a Riemannian Hilbert manifold. Of course, the holonomy pseudogroup makes sense in this set-up. Here, an open transversal is an embedded space that is locally homeomorphic to a transversal associated to a foliated chart via a projection map.

We consider functions that are $C^2$ on the leaves whose tangential derivatives of order $\le2$ are continuous on the ambient space. The functions satisfying the above property are called $C^2$, and they form a linear space denoted by $C^{2}(\FF)$. For a function $f\in C^{2}(\FF)$, we set $\Crit_{\FF}(f)=\bigcup_{L\in\FF}\Crit(f|_L)$. The definition of {\em $C^r$ Hilbert laminations\/} and their $C^r$ functions are analogous.

Since any separable Hilbert space admits $C^\infty$ partitions of unity, the transverse model is paracompact and Hausdorff and we consider a locally finite atlas, we have the following.

\begin{prop}\label{p:existence of a partition of 1}
Every open cover of a Hilbert lamination of class $C^k$ admits a subordinate partition of unity of class $C^k$.
\end{prop}

As a consequence of the Proposition~\ref{p:existence of a partition of 1} and following the same argument of \cite{Candel-Conlon, Alvarez-Candel}, we have

\begin{prop} \label{p:lamination in Hilbert space}
If $(X,\FF)$ is a lamination of class $C^k$, then there is a smooth $C^k$ embedding $\varphi$ of $X$ in the separable real Hilbert space $\mathbb{E}$. Moreover, a given metric tensor along the leaves can be extended to a metric tensor on $\mathbb{E}$.
\end{prop}

Let $(X,\FF)$ be a usual lamination (of finite dimension), of class $C^k$ with $k\ge1$, embedded in the Hilbert space $\mathbb{E}$. The restriction of the embedding to each leaf is not an embedding, but only an injective immersion. The smoothness of $\FF$ being at least $C^1$ implies that the map which assigns to a point $x \in X$ the subspace $T_x\FF$ of $\mathbb{E}$ is continuous (as a map of $\FF$ into the space of $n$-dimensional subspaces of $\mathbb{E}$). It follows that if $F$ is a subspace complementary to one $T_xX$ in $\mathbb{E}$, then it is also complementary to $T_yX$ for $y$ close to $x$. The key point is that each tangent space $T_xX$ is a finite dimensional subspace of $\mathbb{E}$. Hence it is closed and has an orthogonal complement.

\begin{thm}[\cite{Alvarez-Candel}]\label{t:Hilbert Orthogonal}
Let $X$ be a lamination of finite dimension embedded in $\mathbb{E}$ as above, of class $C^2$. Let $i:L\to\mathbb{E}$ denote the inclusion of a leaf $L$ in $X\subset\mathbb{E}$. Then there is a vector bundle $\pi : N \to L$ and a neighborhood $W$ of the zero section of $N$ such that the following properties hold:
\begin{itemize}

\item[(i)] The map $i:L\to\mathbb{E}$ extends to a local diffeomorphism $\varphi:W\to\mathbb{E}$;

\item[(ii)] there is a laminated subspace $Y \subset W$, of the same dimension as $X$, having $L$ as a
leaf and transverse to the fibers of $N$; and

\item[(iii)] as foliated spaces, $Y = \varphi^{-1}(X \cap \varphi(W ))$, and the restriction of $\pi$ to each leaf of $Y$ is a local diffeomorphism into $L$.

\end{itemize}
\end{thm}

\begin{defn}[Tangential isotopy]
Let $M$ be a $C^r$ Hilbert manifold, $r\geq 1$. A $C^q$ {\em isotopy\/} on $M$, $1\leq q\leq r$, is a
$C^q$ differentiable map $\phi:M\times\R\to M$ such that $\phi_t=\phi(-,t):M\to M$
is a diffeomorphism for all $t\in [0,1]$ and $\phi_0=\id_M$. If
$(X,\FF)$ is a $C^{r}$ Hilbert lamination, a $C^q$ {\em
tangential isotopy\/} on $(X,\FF)$ is a leafwise $C^{q}$ map
$\phi:X\times \R\to X$ such that the functions $\phi_t=\phi(-,t):X\to X$ are
homeomorphisms for all $t$, and $\phi$ restricts to usual isotopies on the leaves of $\FF$.
\end{defn}

\begin{rem}\label{r:openinvtop}
Let $\phi$ be a tangential isotopy on $X$ and let
$U\subset X$ be an open set. Then
$\Cat(U,\FF)=\Cat(\phi_t(U),\FF)$ for all $t\in\R$.
\end{rem}

\begin{exmp}[Construction of a tangential isotopy \cite{Palais}]\label{e:vectorfield}
A tangential isotopy can be constructed on a Hilbert manifold by using a $C^1$ tangent vector field $V$. There exists a flow $\phi_t(p)$ such that $\phi_0(p)=p$, $\phi_{t+s}(p)=\phi_t(\phi_s(p))$ and $d\phi_t(p)/dt=V(\phi_t(p))$. From the way of obtaining $\phi$ \cite{Palais,Crandall-Pazy}, it follows that the same kind of construction for a $C^1$ tangent vector field on a measurable Hilbert lamination $(X,\FF)$ induces a tangential isotopy on $(X,\FF)$.

Now we obtain a tangential isotopy from the gradient flow of a differentiable map. It will be modified by a control function $\alpha$ in order to have some control on the deformations induced by the corresponding isotopy.
Let $\nabla f$ be the gradient tangent vector field of $f$; i.e., the
unique tangent vector field satisfying $df(v)=\langle v,\nabla
f\rangle$ for all $v\in T\FF$. Take the $C^{1}$ vector field $V=-\alpha(|\nabla f|)\,\nabla f$, where $\alpha:[0,\infty)\to\R^+$ is $C^\infty$, $\alpha(t)\equiv 1$ for $0\leq t\leq 1$, $t^2\alpha(t)$ is monotone non-decreasing and $t^2\alpha(t)=2$ for $t\geq 2$. The flow $\phi_t(p)$ of $V$ is defined for $-\infty<t<\infty$ \cite{Schwartz}, and it is called the {\em modified gradient flow\/}.
\end{exmp}

Let us define a partial order relation ``$\ll$'' for the critical
points of $f$. First, we say that $x<y$ if there exists a regular
point $p$ such that $x\in\alpha(p)$ and $y\in\omega(p)$, where
$\alpha(p)$ and $\omega(p)$ are the $\alpha$- and $\omega$-limits of
$p$. Thendefine $x\ll y$ if there exists a finite sequence of critical
points, $x_1,\dots,x_n$, such that $x < x_1 < \dots< x_n < y$.

Let $\gamma(p)$ respectively denote the flow orbit of each point $p$; that is, the trace of the curve $\phi_t(p)$ with $t\in\R$.

\begin{lemma}\label{l:toplowaproxim}
Let $T\subset X$ be a closed set into an open transversal, and let $V$ be an open set such that $T\subset V$. Then there exists a \FF-categorical open set $U$ containing $T$ such that
$T\subset U\subset V$.
\end{lemma}

\begin{proof}
An open transversal is homeomorphic to an open set of a associated transversal of a foliated chart via the projection map. Hence it is easy to prove the lemma locally. We can extend the proof by using the locally finiteness of the atlas and the existence of partitions of unity in order to give a global tangential contraction.
\end{proof}

We restrict our study to the case where the leafwise critical points are isolated on the leaves. Therefore the set of leafwise critical points is a transverse object.

\begin{lemma}\label{pr:measurableflow}
Suppose that $\Crit_\FF(f)$ is a transverse set. The modified gradient flow $\phi$ \upn{(}see Example~\ref{e:vectorfield}\upn{)}
satisfies the following properties:
\begin{itemize}
  \item[(i)] The flow runs towards lower level sets of $f$, i.e., $f(p)\geq f(\phi_t(p))$ for $t>0$.

  \item[(ii)]The invariant points of the flow are just the critical points of $f$.

  \item[(iii)] A point is critical if and only if $f(\phi_t(p))=f(p)$ for some $t\neq 0$.

  \item[(iv)] The points in the $\alpha$- and $\omega$-limits are critical points if they are non empty.

\end{itemize}
\end{lemma}

\begin{proof}
These properties can be proved in each leaf, considered as a $C^2$
Hilbert manifold, where~(i),~(ii) and~(iii) follow from the work of J.~Schwartz \cite{Schwartz}.

Under these conditions, the $\alpha$- and $\omega$-limits are
connected sets that consist of critical points if they are non-empty (by using~(i),~(ii) and~(iii)). If $\omega(p)$ is infinite, then all of its points
are non-isolated, contradicting the assumption.
\end{proof}

\begin{defn}[Critical sets]\label{d:partcrit}
Suppose that $f:\FF\to\R$ is a $C^2$ function such that for each leafwise critical point $p$, either $p$ is a relative minimum, or there exists another critical point $x$ such that $p<x$. Define $M_0,M_1,\dots$ inductively by
\begin{align*}
M_0&=\{\,x\in \Crit_\FF(f)\mid\not\exists\ y\ \text{such that}\ x\ll y\,\}\;,\\
M_i&=\{\,x\in\Crit_\FF(f)\mid\forall y\,\ x\ll y\ \Rightarrow\ y\in M_0\cup\dots\cup M_{i-1}\,\}\;.
\end{align*}
Clearly, $M_0$ contains all relative minima on the leaves. We also
set $C_0(f)=M_0$ and $C_i(f)=M_i\setminus(M_0\cup\dots\cup
M_{i-1})$. Observe that, if $x\in\omega(p)\cap C_i(f)$ for
some $i$, then $\omega(p)\subset C_i(f)$. There is an analogous
property for the $\alpha$-limit. The notation $C_i$ will be used
if there is no confusion and they will be called the {\em critical sets\/} of $f$. Let $p\ll p^*$. Then $i_{p^*}< i_p$,
where $i_{p}$ and $i_{p^*}$ are the indexes such that $p\in
C_{i_p}$ and $p^*\in C_{i_{p^*}}$. Observe that the critical sets are $\sg$-compact.
\end{defn}

The definition of a Palais-Smale condition is needed for this section. For Hilbert laminations, it could be adapted by taking functions that satisfy the Palais-Smale condition on all (or almost all) leaves. But this is very restrictive because it would mean that the set of relative minima meets each leaf in a relatively compact set (which is non-empty when $f$ is bounded from below), and therefore there would exist a complete transversal meeting each leaf at one point. This would be a very restrictive condition on the foliation. Thus, instead, we use the following weaker condition.

\begin{defn}\label{d:PStopological}
An {\em $\omega$-Palais-Smale\/} (or simply, {\em $\omega$-PS\/}) function is a function $f\in
C^{2}(\FF)$ such that all of its critical sets are closed (in the ambient topology), for any $p\in\Crit_\FF(f)$, the set $\{\,x\in\Crit_\FF(f)\mid p\ll x\,\}$ is compact, and this set is empty if and only if $p$ is a relative minimum and any flow line of $\phi$ has a non-empty $\omega$-limit. An {\em $\alpha$-Palais-Smale\/} (or simply, {\em $\alpha$-PS\/}) function is defined analogously by taking the set $\{\,x\in\Crit_\FF(f)\mid x\ll p\,\}$.
\end{defn}

Of course, $f$ is $\omega$-PS if and only if $-f$ is $\alpha$-PS. The set of relative minima of a $\omega$-PS function bounded from below is always non-empty in any leaf. In the case of a single manifold and $f$ with isolated critical points, every bounded from below Palais-Smale function is an $\omega$-PS function.

\begin{rem}\label{r:goodcriticalsets}
In a $\omega$-PS function, each critical set is a closed set in an open transversal. If there exists a branching point in a critical set $C_i$ then we can use the mountain pass theorem \cite{Jabri} to obtain that some $C_j$ is not closed for $j\neq i$.
\end{rem}

\section{Main theorem}

\begin{thm}\label{t:TanCrit}
Let $(X,\FF)$ be a Hilbert lamination endowed with a Riemannian metric on the leaves varying continuously on the ambient space, and let $f$ be an $\omega$-PS function. Suppose that $\Crit_\FF(f)$ meets each leaf in a discrete set. Then
$\Cat(\FF)\leq \#\{\text{critical sets of}\ f\}\;.$
\end{thm}

\begin{proof}
By Remark~\ref{r:goodcriticalsets}, we can apply Lemma~\ref{l:toplowaproxim}. Since the critical sets are closed and disjoint, there exists a family of mutually disjoint open sets,
$\{U_i\}$  ($i\in\N\cup\{0\}$), where each $U_i$ contains $C_i$, and
such that each $U_i$ is tangentially categorical. We have $\overline{U_i}\subset \widetilde{U_i}$, where $\widetilde{U_i}$ is an \FF-categorical open set that does not contain regular points $p$ with $\omega(p)\in C_j$ for $j>i$. 

Let $\phi$ be the modified gradient flow (see Example~\ref{e:vectorfield}); thus $(X\setminus \Crit_\FF(f),\phi)$ is a $1$-dimensional lamination and we can apply Proposition~\ref{p:lamination in Hilbert space} and Theorem \ref{t:Hilbert Orthogonal} to it. 

Let $p$ be a relative minimum ($p\in C_0$), by the continuity of the gradient flow we can find a foliated chart around $p$ where each plaque has at least a relative minimum, and these relative minima converge to $p$ when the plaques approach the plaque of $p$. Moreover each plaque has only one relative minima if the foliated chart is chosen small enough since, if infinitely many of them have more than one relative minima, then there exists a sequence of points in $C_1(f)$ converging to $p$ by the mountain pass theorem \cite{Jabri}, which contradicts the assumption on the critical sets to be closed. Therefore $C_0(f)$ is an (embedded) open transversal meeting each leaf in a discrete set.  So $U_0$ can be chosen to be a tube around $C_0$; in particular, it tangentially contracts to $C_0$.

Let $U'_0=\bigcup_{n\in\N}\phi_{-n}(U_0)$, which is open since $C_0(f)$ consists of relative minima, and it is tangentially categorical since the flow $\phi$ contracts all the points of $U'_0$ to $C_0$. The set $X_1=X\setminus U'_0$ is closed. The critical set $C_1(f)$ equals the set of relative minima of the restriction $f|_{X_1}$. Notice that $X_1$ consists of critical points outside of $C_0(f)$ and regular points connecting these critical points according to the definition of the relation ``$\ll$''. Let $F_1=U_1\cap X_1$ and $F'_1=\bigcup_{n\in\N}\phi_{-n}(F_1)$. The set $F'_1$ is open in $X_1$ and closed in $X$. The set $F'_1\setminus C_1(f)$ is $\phi$-saturated and closed in $X\setminus \Crit_\FF(f)$. There is an open set $U'_1$ containing $F'_1$ such that there exists a measurable deformation $H$ satisfying $H(U'_1\times\{1\})\subset \widetilde{U_1}$. 

The set $U'_1$ is defined as follows. Observe that the flow lines of $(X\setminus \Crit_\FF(f),\phi)$ are embedded (and not only immersed) in $\mathbb{E}$ by the properties of the gradient flow. We consider $F'_1\setminus C_1(f)$ as a subset of $X\setminus \Crit_\FF(f)$ and embedded in the Hilbert space $\mathbb{E}$. Consider the open subset $V_1=\bigcup_\gamma W(\gamma)$ of $X\setminus \Crit_\FF(f)$, where $\gamma$ runs in the family of flow lines in $F'_1\setminus C_1$ and $W(\gamma)$ is a tubular neighborhood of $\gamma$ provided by Theorem \ref{t:Hilbert Orthogonal}.  By the Lindel\"off property, we can assume that $V_1$ is a countable union of tubular neighborhoods of flow lines: $V_1=\bigcup_{n\in\N}W(\gamma_n)$. Let $\pi_n:W(\gamma_n)\to\gamma_n$ be the projections given by the same Theorem \ref{t:Hilbert Orthogonal}. We can suppose also that the family $(\gamma_n)_{n\in\N}$ is locally finite by paracompactness. Let $\lambda_n:V_1\to[0,1]$, $n\in\N$, be a partition of unity associated to the $\{W(\gamma_n)\}_{n\in\N}$. For each $x\in V_1$, let $I(x)\subset\N$ be the set of numbers $n$ such that $x\in W(\gamma_n)$. The isotopy $\phi|_{F'_1}$ contracts $F'_1$ to $C_1$. We extend the deformation ${\phi}|_{F'_1\setminus C_1(f)}$ to the neighborhood $V_1$. This extension can be defined if we consider our embedding of $X\setminus\Crit_\FF(f)$ in $\mathbb{E}$: for $x\in V_1$, $t\in\R$ and $n\in I(x)$, let $r(x,t,n)$ be the unique positive real number such that $\phi_{r(x,t,n)}(x)=\gamma(x)\cap \pi_n^{-1}(\phi_t(\pi_n(x)))$. Let $:V_1\times\R\to X$ be the continuous map defined by $H(x,t)=\phi_{s(x,t)}(x)$, where
$$
  s(x,t)=\sum_{k\in I(x)}\lambda_k(x)\,r(x,t,k)\;.
$$ 
For $x\in V_1$ and $t\in\R$, there exists $k_1,k_0\in I(x)$ such that $r(x,t,k_1)\leq s(x,t)\leq r(x,t,k_0)$. It is clear that $\lim_{t\to \infty}\phi_{r(x,t,n)}\subset \overline{U_1}\subset\widetilde{U_1}$ for all $n\in\N$.  Let $p\in C_1$ and let $x\in F'_1\setminus C_1$ with $\omega(x)=p$. By the continuity of $\phi$, there exists a neighborhood $U(p)$ of $p\cup\gamma(x)$, a tubular neighborhood $V(\gamma(x))$ of $\gamma(x)$ contained in $V_1$ and $T\in\R$ such that $\phi_{r(y,t,n)}\subset U(p)$ for all $y\in V(\gamma(x))$, $t>T$ and $n\in I(y)$. Therefore  $\lim_{t\to \infty}H(x,t)\subset \overline{\bigcup_{p\in C_1} U(p)}\subset\overline{U_1}\subset\widetilde{U_1}$ for all $x\in\bigcup_{\gamma\subset F'_1\setminus C_1}V(\gamma)\subset V_1$. Then the open subset $V'_1=\bigcup_{\gamma\subset F'_1\setminus C_1}V(\gamma)\subset X$ is $\FF$-categorical (by a standard change of parameter). Finally,  if $\widetilde{U_1}$ is small enough, $U'_1=V'_1\cup \widetilde{U_1}$ is $\FF$-categorical by a telescopic argument \cite{Hatcher}and $F'_1\subset U'_1$.

This process can be done inductively by taking $M_n=M\setminus (U'_0\cup\bigcup_{i=1}^{n-1}F'_i)$, and by using the same trick to define $U'_n$ observing that $C_n(f)$ is the set of relative minima of $M_n$.
\end{proof}

\begin{question}
We can ask if the same results are also true when the critical sets are not closed. We are greatly convinced that the answer is affirmative, but the proof seems to be much more difficult.
\end{question}

\section{More examples and possible applications}

\begin{exmp}[Revisiting a classical result]
One of the most important application of the Lusternik-Schnirelmann category is to give lower estimations of the number of the possible solutions to a variational problem involving a Palais-Smale map (like energy functionals). For instance, the lower estimation for the number of closed geodesics in compact Riemannian manifolds \cite{Lusternik-Schnirelmann,Ballman,Ballman-Thorbergsson-Ziller}.

The general and correct application to this problem was done by W. Klingenberg \cite{Klingenberg}, or W. Ballman, G. Thorbergsson and W. Ziller \cite{Ballman-Thorbergsson-Ziller}. They use Lusternik-Schnirelmann general theory in order to obtain this lower estimation, but in fact they use the cohomology cuplength in their computations. There are two difficulties when we want to apply classical Lusternik-Schnirelmann category to this problem:
\begin{itemize}
\item [(i)] Let $\Theta(M)$ be the Hilbert manifold of parametrized absolutely continuous closed curves with finite energy on $M$. Then different iterations of the same closed geodesic represent different critical points of the energy functional.

\item [(ii)] The action of $O(2)$ on $\Theta(M)$, via linear change of parametrizations, is not free in general, only continuous. Therefore we have to apply Lusternik-Schnirelmann theory to the quotient space $\Theta(M)/O(2)$ of unparametrized closed curves.
\end{itemize}

Now, remember our improvement of the classical Lusternik-Schnirelmann theorem (Theorem~\ref{t:TanCrit}). Let $J:\Theta(M)\to\R$ be the kinetic energy functional, $J(\sg)=1/2\int_{S^1}|\sg'(t)|^2\,dt$, which is Palais-Smale. Unfortunately, its critical points are not topological transversals if we do not consider unparametrized curves. But the action of $O(2)$ preserves the energy and the gradient flow of $J$. Therefore we can adapt our improvement of the classical theorem to this problem when the critical orbits are separated. In this case, our formulation would provide three advantages:
\begin{itemize}
\item [(i)] All critical points in the same critical $O(2)$-orbit belong to the same critical set.

\item [(ii)] Iterated closed geodesics (where the action of $O(2)$ is not free) seem to be in the same critical set of the corresponding prime geodesic.
    
\item [(iii)] Our $\omega$-PS condition is weaker than the usual Palais-Smale condition. Thus we hope to use the improved theorem in suitable subsets of $\Theta(M)$ where the restriction of $J$ is not Palais-Smale in the usual sense.
\end{itemize}

We expect that these advantages could simplify the classical proofs and may improve some results on the general problem of the existence of closed geodesics.
\end{exmp}

\begin{exmp}[Tool to obtain upper bounds of the (tangential) category]
Since $\omega$-PS condition is weaker than the Palais-Smale condition, it is expected that our theorem will provide upper bounds for the tangential and usual LS category. So it is a tool to obtain higher estimations for these categories.
\end{exmp}

\begin{exmp}[Foliated setting]
Considering a usual lamination $\FF$, the space of absolutely continuous closed curves with finite energy on leaves, which is denoted $\Theta(\FF)$. We expect that a structure of  Hilbert lamination can be realized and the energy functional $J:\Theta(\FF)\to\R$ will be $\omega$-PS (but, in general, it would not be Palais-Smale on leaves). Its critical sets give information about the evolution of the closed geodesics on leaves relative to the leafwise gradient flow of $J$. As before, we have a natural action of $O(2)$ with the same properties as above.
\end{exmp}

\begin{exmp}[Hilbert torus]
Let $l^2(\R)=\{(x_n)\in\R\mid \sum_n |x_n|^2<\infty\}$ with the usual Hilbert product and consider the map $\pi:l^2(\R)\to \prod_{n=1}^\infty S^1=T^\infty$, $(x_n)\mapsto (\exp(2\pi i x_n))$. Consider the topology on $Y=\pi(l^2(\R))$ such that $\pi:l^2(\R)\to Y$ is a surjective local homeomorphism; in this way, a structure of Hilbert manifold is induced on $Y$. Even with this topology, $\pi_1(Y)=\bigoplus_{n=1}^\infty \Z$ and the cuplength is infinite, yielding $\Cat(Y)=\infty$.

A Hilbert lamination can be defined with the suspension of the homomorphism $\pi_1(Y)\to\operatorname{Diff}(S^1)$ given by $e_n\mapsto R_n$, where the elements $e_n$ form the standard base of $\bigoplus_{n=1}^\infty \Z$ and $R_n:S^1\to S^1$ is a rotation of the circle. Thus we have a structure of Hilbert lamination on the suspension $X=l^2(\R)\times_{(R_n)}S^1$. Of course, the usual Lebesgue measure on $S^1$ is an invariant measure for this lamination. Since the cuplength of $Y$ is infinite, it is clear that the leafwise cuplength of $X$ is also infinite. Therefore $\Cat(X)=\infty$ (independently of the choice of the rotations $R_n$).

Thus, in this Hilbert lamination, the number of critical sets of any $\omega$-PS map is always infinite.
\end{exmp}

\begin{ack}
This paper contains part of my PhD thesis, whose advisor is Prof.
Jes\'{u}s A. \'{A}lvarez L\'{o}pez.
\end{ack}

\end{document}